\documentclass[
final
]{dmtcs-episciences}

\usepackage[utf8]{inputenc}
\usepackage[numbers,sort&compress]{natbib}

\usepackage{amsmath}
\usepackage{amssymb}
\usepackage{amsthm}
\usepackage{mathtools}

\usepackage{cleveref}

\theoremstyle{plain}
\newtheorem{thm}{Theorem}
\newtheorem{lemma}[thm]{Lemma}
\newtheorem{prop}[thm]{Proposition}
\newtheorem{qn}[thm]{Question}

\theoremstyle{definition}
\newtheorem{defn}[thm]{Definition}

\theoremstyle{remark}
\newtheorem{rem}[thm]{Remark}
\newtheorem*{ex}{Example}

\numberwithin{equation}{section}
\numberwithin{thm}{section}

\DeclareMathOperator{\supp}{supp}
\DeclareMathOperator{\adim}{adim}
\DeclareMathOperator{\bdim}{bdim}
\DeclareMathOperator{\Z}{\mathbb{Z}}
\DeclareMathOperator{\s}{\mathcal{S}}
\DeclareMathOperator{\B}{\mathcal{B}}

\author[R. Madhukara]{Rachana Madhukara}
\title[Adjacency and Broadcast Dimension]{Adjacency and Broadcast Dimension of Grid and Directed Graphs}
\affiliation{
  Department of Mathematics, Massachusetts Institute of Technology, Cambridge, MA, USA}

\begin{document}
\publicationdata{vol. 28:3}{2026}{9}{10.46298/dmtcs.11737}{2023-08-17; 2023-08-17; 2024-04-19}{2026-08-12}
\maketitle
\begin{abstract}
Let $G$ be a simple undirected graph. A function $f : V(G) \to \mathbb{Z}_{\geq 0}$ is a \textit{resolving broadcast} of $G$ if for any distinct $x, y \in V(G)$, there exists a vertex $z \in V(G)$ with $f(z) > 0$ such that $\min \{ d(z, x), f(z)+1 \} \neq \min \{ d(z, y), f(z)+1 \}$.  The \textit{broadcast dimension} $\text{bdim}(G)$ of $G$ is the minimum of $\sum_{v \in V(G)} f(v)$ over all resolving broadcasts $f$ of $G$.  Similarly, the \textit{adjacency dimension} $\text{adim}(G)$ of $G$ is the minimum of $\sum_{v \in V(G)} f(v)$ over all resolving broadcasts $f$ of $G$ where $f$ takes values in $\{0,1\}$. These parameters are defined analogously for directed graphs by considering directed distances.

We partially resolve a question of Zhang by obtaining precise bounds for the adjacency dimension of certain Cartesian products of path graphs, namely $\text{adim}(P_2 \square P_n)$ and $\text{adim}(P_3 \square P_n)$. Additionally, we study the behavior of adjacency and broadcast dimension on directed graphs. First, we explicitly calculate the adjacency dimension of a directed complete $k$-ary tree, where every edge is directed towards the leaves. Next, we prove that $\text{adim}(\vec{G}) = \text{bdim}(\vec{G})$ for some particular directed trees $\vec{G}$. Furthermore, we show that $\text{bdim}(G)$ can be as large as an exponential function of $\text{bdim}(\vec{G})$ or as small as a logarithmic function of $\text{bdim}(\vec{G})$.  
\end{abstract}

\section{Introduction and statement of results}

The \textit{metric dimension} of a graph is a well-studied graph property that was introduced by Slater \cite{slater} in 1975 and then independently introduced by Harary and Melter \cite{harary} in 1976. Metric dimension has gained prevalence because of its application to many different areas, even outside of mathematics. For example, Chartrand, Eroh, Johnson, and Oellermann \cite{chartrand2000resolvability} study metric dimension in the context of chemistry by considering chemical compounds as graphs. Other papers have considered metric dimension in the context of pattern recognition and image processing \cite{melter1984metric} and as a tool to find strategies for the game Mastermind \cite{chvatal1983mastermind}.    

While there are many equivalent ways of defining metric dimension, we present the original definition introduced by Slater. Let $G = (V(G), E(G))$ be a finite, simple, and undirected graph. For any two vertices $u, v \in V(G)$, we define $d(u, v)$ to be the length of the shortest path between two vertices $u, v \in V(G)$ if they lie in the same connected component of $G$, or $\infty$ otherwise. Moreover, we define $d_k(u, v) \coloneqq \min(d(u, v), k+1)$ for a positive integer $k$ and vertices $u, v \in V(G)$. 

\begin{defn}
A subset $S \subseteq V(G)$ is called a \textit{resolving set} of the graph $G$ if for any distinct $x, y \in V(G)$ there exists a vertex $z \in S$ such that $d(z, x) \neq d(z, y)$. The \textit{metric dimension} $\dim(G)$ of $G$ is the minimum cardinality of a resolving set of $G$. 
\end{defn}

Following Khuller, Raghavachari, and Rosenfeld \cite{khuller1996landmarks}, imagine a robot navigating along the vertices of $G$. At certain vertices, there exist distinct landmarks and the robot can sense its distance from each landmark. The metric dimension of $G$ is the minimum number of landmarks that must be installed in order for the robot to uniquely determine its position.   

While studying the problem of robot navigation and the problem of determining the metric dimension of lexicographic graph products, Jannesari and Omoomi introduced the concept of \textit{adjacency dimension} in 2012 \cite{jannesari2012metric}.

\begin{defn}
A subset $A \subseteq V(G)$ is called an \textit{adjacency resolving set} of $G$ if for any distinct $x, y \in V(G)$ there exists a vertex $z\in A$ such that $d_1(z, x) \neq d_1(z, y)$. The \textit{adjacency dimension} $\adim(G)$ of $G$ is the minimum cardinality of an adjacency resolving set of $G$.   
\end{defn}

It should be noted that an adjacency resolving set is similar to the notion of a \textit{locating-dominating set}. Introduced by Slater in 1987 \cite{slater1987domination}, locating-dominating sets are well studied \cite{bertrand2004identifying, carson1995generalized, charon2002identifying, charon2003minimizing, charon2007extremal, colbourn1987locating, finbow1988locating, haynes2013fundamentals, honkala2004locating, rall1984location, slater1988dominating, slater1995locating}. Here we present the definition from \cite{honkala}. First, we define the neighborhood of a vertex $v$ as $N(v) \coloneqq \{ w \in V(G) : vw \in E(G) \}$ and the $k$-\textit{neighborhood} of a vertex $v$, denoted $N^k[v]$, as the set of vertices $u$ such that $d(v, u) \leq k$. Notice that $N^1[v]$ is the closed neighborhood $N(v) \cup \{v\}$. Then we have the following definition.   

\begin{defn}[{\cite{honkala}}]\label{LD}
Consider a subset $C \subseteq V(G)$ and define the \textit{identifying set} ($I$\textit{-set}) of a vertex $v$ as $I(v) \coloneqq N^1[v] \cap C$. Then $C$ is called a \textit{locating-dominating set} of $G$ if the sets $I(v)$ are non-empty and $I(v) \neq I(u)$ for all $v, u \in V(G) \setminus C$, $v \neq u$. The \textit{locating-dominating number} $LD(G)$ of $G$ is the minimum cardinality of a locating-dominating set of $G$. 
\end{defn}

The difference between an adjacency resolving set and locating-dominating set lies in the fact that for adjacency resolving sets we allow the $I$-set of at most one vertex in $V(G)$ to be empty. This gives us the following trivial bound.  

\begin{prop}\label{boundprop}
We have that $LD(G)-1 \leq \adim(G) \leq LD(G)$.
\end{prop} 

Because of the similarities between adjacency resolving sets and locating-dominating sets, some results for adjacency resolving sets, such as density results, are already known. 

\begin{defn}[{\cite{slater2002fault}}]
The parameter $LD\%$ is defined by $$LD\% \coloneqq \min\left\{ \limsup \frac{|S \cap N^k[v]|}{|N^k[v]|} : S \text{ is locating-dominating}\right\},$$ where the minimum is taken over all $v \in V(G)$. 
\end{defn}

For the set of integers $\Z$, let $\Z \times \Z$ denote the graph with vertex set $V(\Z \times \Z) = \{(i, j) : i, j \in \Z \}$ and edge set defined by $N((i, j)) = \{ (i, j-1), (i, j+1), (i-1, j), (i+1, j)\}$. Then we have that

\begin{thm}[{\cite[Theorem 3]{slater2002fault}}]\label{slaterdensity}
$LD\%(\Z \times \Z) = 3/10$.
\end{thm}

Within the problem of robot navigation, adjacency dimension is the minimum number of landmarks that a robot needs in order to uniquely determine its location from only landmarks that are on the vertex it is on or adjacent to the vertex it is on. The question of robot navigation on a graph was generalized when Geneson and Yi \cite{geneson2020broadcast} introduced the \textit{broadcast dimension} of a graph. 

\begin{defn}
A function $f : V(G) \to \Z_{\geq 0}$ is a \textit{resolving broadcast} of $G$ if for any distinct $x, y \in V(G)$ there is a vertex $z \in \supp(f) \coloneqq \{v \in V(G) : f(v) > 0\}$ such that $d_{f(z)}(z, x) \neq d_{f(z)}(z, y)$. The \textit{broadcast dimension} $\bdim(G)$ of $G$ is the minimum value of $\sum_{v \in V(G)} f(v)$ over all resolving broadcasts $f$ of $G$. 
\end{defn}

Here the application to robot navigation is the following: imagine that transmitters with varying ranges are located at different vertices of a graph. Each transmitter has a cost $k \in \Z^+ \cup \{0\}$ which is exactly the range $k$ that it can transmit. A robot can determine its distance to transmitters it is in the range of. The minimum total cost of transmitters needed for the robot to uniquely determine its location is the broadcast dimension of the graph.  

Geneson and Yi prove an asymptotic lower bound on the adjacency and broadcast dimensions of graphs of order $n$. Moreover, they show that the lower bound must be asymptotically tight by using a family of graphs introduced in  \cite{zubrilina2018edge}.

\begin{thm}[{\cite[Theorem 3.7]{geneson2020broadcast}}]
For all graphs $G$ of order $n$, we have $$n \geq \adim(G) \geq \bdim(G) = \Omega(\log n).$$
\end{thm}

In 2020, Zhang \cite{zhang2020broadcast} improved this lower bound on the broadcast dimension for forests of order $n$ and also showed that the new lower bound is asymptotically tight. 

\begin{thm}[{\cite[Theorem 1.3]{zhang2020broadcast}}]
For all forests $G$ of order $n$, we have $\bdim(G) = \Omega(\sqrt{n}),$ and this lower bound is asymptotically tight. 
\end{thm}

Zhang also proved many other results on adjacency and broadcast dimension and posed several open questions. 

\begin{qn}[{\cite[Question 7.2]{zhang2020broadcast}}]
What is the broadcast dimension of the grid graph $P_m\square P_n$, where $\square$ denotes the Cartesian product of two graphs?
\end{qn}

We partially resolve Zhang's question with the following two theorems: 

\begin{thm}\label{2block}
For $n \geq 2$, we have that $\lceil (3n-1)/4 \rceil - 1 \leq \adim(P_2 \square P_n) \leq \lceil (3n-1)/4 \rceil.$
\end{thm}

\begin{thm}\label{3block}
If $n \equiv 1 \bmod 3$, then $$n \leq \adim(P_3 \square P_n) \leq n+1.$$ Otherwise, $$n-1 \leq \adim(P_3 \square P_n) \leq n.$$
\end{thm}

Note that Theorem \ref{slaterdensity} gives us the density of the adjacency resolving set in the infinite grid case. 

We also prove a lower bound on the adjacency dimension of a graph based on the maximum degree of the graph. 

\begin{thm}\label{maxdegree}
Let $G$ be a graph with maximum degree $\Delta$. Then $$\adim(G)\geq\frac{2(|V(G)|-1)}{\Delta+3}.$$
\end{thm}

Lastly, we introduce the concept of adjacency and broadcast dimension for directed graphs. In this context, we compute the adjacency dimension of a directed rooted complete $k$-ary tree. We also show that the adjacency and broadcast dimensions agree for some particular directed trees.  

\begin{thm}\label{allthesame}
Let $\vec{G}$ be a directed rooted tree with all edges directed away from the root. Then $\adim(\vec{G}) = \bdim(\vec{G})$. 
\end{thm}

The paper is structured in the following way: in Section \ref{prelim}, relevant terminology and notation is introduced. Additionally, we state some key preliminary results on metric, adjacency, and broadcast dimension. In Section \ref{maxdeg}, we prove Theorem \ref{maxdegree}. In Section \ref{grids}, the main results on grid graphs are proven. In Section \ref{dirgraph}, the results on directed graphs are proven. 

\section{Preliminaries}\label{prelim}

We introduce relevant terminology and notation in this section. Let $G = (V(G), E(G))$ be a finite, simple, undirected graph and $f : V(G) \to \Z_{\geq 0}$ be a function. Recall that $d_k(u, v) \coloneqq \min(d(u, v), k+1)$.  

\begin{defn}
A vertex $z\in \supp(f)$ \textit{resolves} a pair of distinct vertices $x, y \in V(G)$ if $$d_{f(z)}(z, x) \neq d_{f(z)}(z, y).$$
\end{defn}

\begin{defn}
A pair of distinct vertices $x, y \in V(G)$ is \textit{differentiated} if there exists a vertex $z\in \supp(f)$ such that $d_{f(z)}(z, x) \neq d_{f(z)}(z, y)$. In other words, two vertices are differentiated if there is some vertex that resolves them. 
\end{defn}

In this language, a function $f : V(G) \to \Z_{\geq 0}$ is a resolving broadcast of $G$ if all pairs of vertices are differentiated.  

\begin{defn}
A graph $G$ is called \textit{resolved} with respect to a function $f : V(G) \to \Z_{\geq 0}$ if all pairs of vertices in the graph are differentiated. 
\end{defn}

\begin{defn}
An \textit{adjacency resolving broadcast} is a resolving broadcast that only takes on values in $\{0, 1\}$.
\end{defn}

Note that a set is an adjacency resolving set if and only if its indicator function is an adjacency resolving broadcast. 

\begin{defn}
Given an (adjacency) resolving broadcast $f$ for a graph $G$, the \textit{weight} of a vertex $v$ is the value $f(v)$. Additionally, \textit{placing} a weight on a vertex $v$ refers to the act of increasing $f(v)$ by that weight.     
\end{defn}

\begin{defn}
A vertex $z\in \supp(f)$ \textit{reaches} a vertex $v \in V(G)$ with respect to $f$ if $d(z, v) \leq f(z)$. Additionally, the function $f$ \textit{reaches} a vertex $v \in V(G)$ if there is a vertex $z \in \supp(f)$ that reaches $v$ and otherwise we say $v$ is \textit{unseen} by $f$.  
\end{defn}

\begin{rem}\label{unseenrem}
If $f$ is an (adjacency) resolving broadcast, then there is at most one unseen vertex because two unseen vertices cannot be differentiated. 
\end{rem}

We can also understand the difference between the locating-dominating number and adjacency dimension of a graph in terms of these definitions. 

\begin{rem}
The locating-dominating number is very similar to the adjacency dimension, but with the added restriction that all vertices must be reached. In other words, there can be no unseen vertex.
\end{rem}

\section{Lower Bounds for General Graphs}\label{maxdeg}

Before delving into the results on grid and directed graphs, we start by proving a lower bound for the adjacency dimension of a graph based on its maximum degree.  

\begin{proof}[of Theorem \ref{maxdegree}]
Let $S$ be an adjacency resolving set of $G$. The vertices of $G$ can be partitioned into the following sets:

\begin{enumerate}
    \item $T_1 \coloneqq S$
    \item $T_2 \coloneqq \{\text{vertices not in $S$ which are adjacent to a single vertex in $S$}\}$
    \item $T_3 \coloneqq \{\text{vertices not in $S$ which are adjacent to multiple vertices in $S$}\}$
    \item $T_4 \coloneqq \{\text{vertices not in $S$ which are not adjacent to any vertices in $S$}\}$.
\end{enumerate}

Then
\begin{equation}
    |T_1|+|T_2|+|T_3|+|T_4| = |V(G)|.
\end{equation}
All vertices in $|T_2|$ are adjacent to different vertices in $S$. Therefore, 
\begin{equation}\label{eqn1}
    |T_2| \leq |S|.
\end{equation}
Counting the edges incident to the vertices in $S$, each element of $T_2$ contributes one edge and each element of $T_3$ contributes at least two edges. Therefore
\begin{equation}\label{eqn2}
    |T_2| + 2|T_3| \leq \Delta |S|.
\end{equation}
Recalling Remark \ref{unseenrem}, we also have that $|T_4| \leq 1$. Combining Equation \eqref{eqn1} with Equation \eqref{eqn2}, we get that $|T_2|+|T_3| \leq |S|(\Delta+1)/2$. Then using the fact that $|T_1| = |S|$, we have that $$|V(G)| = |T_1|+|T_2|+|T_3|+|T_4| \leq |S|+\frac{|S|(\Delta+1)}{2}+1.$$ Now solving for $|S|$, we have the desired bound.
\end{proof}

\begin{thm}
For any $m \in \mathbb{N}$ and any $\Delta \leq m$ such that $(\Delta-1)m$ is even, there exists a graph $G$ with adjacency dimension $m$ and maximum degree $\Delta$ for which the inequality in Theorem \ref{maxdegree} holds with equality. 
\end{thm}

\begin{proof}
Let $H$ be any $(\Delta-1)$-regular graph on $m$ vertices. We then add a vertex to the middle of each edge of $H$, subdividing the edge. Additionally, we connect a single vertex to each of the original vertices of $H$. Finally, we add one isolated vertex. Call the resulting graph $G$. See Figure \ref{fig:Acquired} for an example with $m=6$ and $\Delta=4$. We have that
\begin{align*}
    |V(G)| &= m+m\cdot\frac{\Delta-1}{2}+m-1\\
    &= m \cdot \frac{\Delta +3}{2}+1.
\end{align*}
Note that $V(H)$ is an adjacency resolving set for $G$. Hence $\adim(G)\leq m$. On the other hand, by Theorem \ref{maxdegree}, we have that $\adim(G)\geq 2(|V(G)|-1)/(\Delta+3) = m$. 
\end{proof}

\begin{figure}[h]
    \centering
    \includegraphics[width=0.9\textwidth]{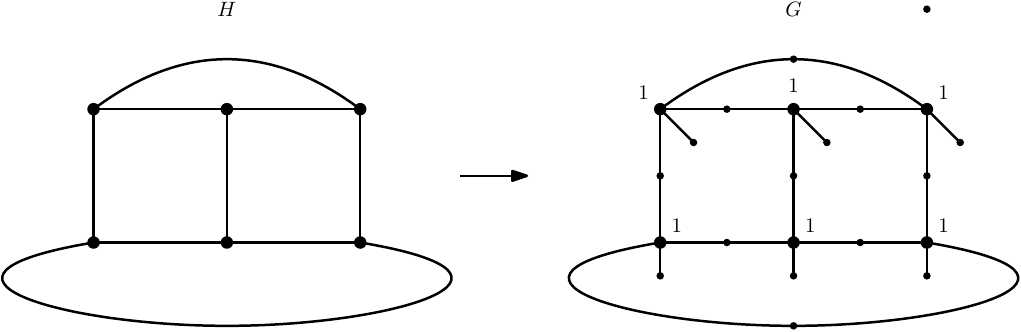}
    \caption{Transformation from $H$ to $G$}
    \label{fig:Acquired}
\end{figure}

\section{Grid Graphs}\label{grids}

In this section, we prove results on grid graphs. In particular, we determine bounds for the adjacency dimension of two different types of grid graphs. We use the locating-dominating number of a graph as a tool to acquire these bounds. 

\begin{defn}
The \textit{Cartesian product} of two graphs $G_1 = (V_1, E_1)$ and $G_2 = (V_2, E_2)$, denoted $G_1\square G_2$, is the following graph:

\begin{enumerate}
    \item The graph $G_1\square G_2$ has vertex set $V_1 \times V_2$.
    \item Two vertices $(u, u')$ and $(v, v')$ are adjacent in $G_1\square G_2$ if and only if either  
            \begin{enumerate}
                \item $u=v$ and $u'$ is adjacent to $v'$ in $G_2$, or 
                \item $u'=v'$ and $u$ is adjacent to $v$ in $G_1$. 
            \end{enumerate}
\end{enumerate}
\end{defn}

Let $P_n$ denote the path graph on $n$ vertices. Then using the Cartesian product, we define $G_{n, m} = P_n \square P_m$ to be the $n \times m$ \textit{grid graph}. Additionally, we impose a coordinate system on a general grid graph $G_{n, m}$. We call the top leftmost vertex $v_1$ and then proceed down the column, labeling the vertices up to $v_n$. Then we continue to the next column and label the vertices in the column from $v_{1+n}$ to $v_{2n}$ and so on until the last column is labeled from $v_{1+n(m-1)}$ to $v_{nm}$.  

\begin{defn}
We define an $(n \times \ell)$-block of a grid graph $G_{n, m}$, where $1 \leq \ell \leq m$, to be the induced subgraph made up of $\ell$ consecutive columns.  
\end{defn}

\subsection{Adjacency Dimension of $G_{2, m}$}

\begin{lemma}\label{beforeafter}
Let $\mathcal{S}$ be a locating-dominating set for $G_{2, m}$ and let $\mathcal{B}$ be a $(2\times 4)$-block. Then $|\mathcal{B} \cap \mathcal{S}| \geq 2$.
\end{lemma}

\begin{proof}
Let $\mathcal{B} = \{v_i, \dots, v_{i+7}\}$. Assume for the sake of contradiction that $|\mathcal{B} \cap \mathcal{S}| \leq 1$. Consider the possible vertices that are in $\mathcal{B} \cap \mathcal{S}$. If $\mathcal{B} \cap \mathcal{S} = \{v_i\}$, then $v_{i+4}$ and $v_{i+5}$ remain undifferentiated (even if $v_{i-2}, v_{i-1}, v_{i+8}, v_{i+9} \in \mathcal{S}$). This is a contradiction. The cases where $\mathcal{B} \cap \mathcal{S}$ is $\{v_{i+1}\}, \{v_{i+6}\}$ or $\{v_{i+7}\}$ are symmetric. Similarly, if $\mathcal{B} \cap \mathcal{S} = \{v_{i+2}\}$, then $v_{i+3}$ and $v_{i+4}$ remain undifferentiated. Once again, we have a contradiction. The cases where $\mathcal{B} \cap \mathcal{S}$ is $\{v_{i+3}\}, \{v_{i+4}\}$ or $\{v_{i+5}\}$ are symmetric.
\end{proof}

See Figure \ref{fig:example2} for an example of how a $(2\times 4)$-block can be resolved with 2 vertices.

\begin{figure}[h]
    \centering
    \includegraphics{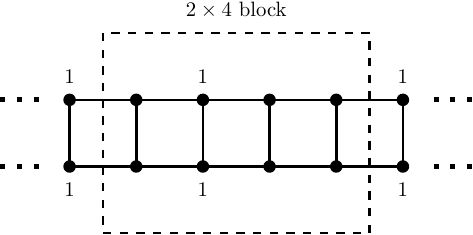}
    \caption{A $(2\times 4)$-block resolved with only 2 vertices}
    \label{fig:example2}
\end{figure}

\begin{lemma}\label{twovertices}
Let $\mathcal{S}$ be a locating-dominating set for $G_{2, m}$. If there exists a $(2 \times 4)$-block $\mathcal{B}$ within $G_{2, m}$ consisting of the vertices $v_i, \dots, v_{i+7}$ such that $|\mathcal{B} \cap \mathcal{S}| = 2$, then $i\neq 1$, $i+7\neq2m$, and $v_{i-2}, v_{i-1}, v_{i+8}, v_{i+9} \in \mathcal{S}$.    
\end{lemma}

\begin{proof}
We consider all possibilities for the placement of the two vertices $w_1, w_2 \in \mathcal{B} \cap \mathcal{S}$ and show that we need $i\neq 1$, $i+7\neq2m$, and $v_{i-2}, v_{i-1}, v_{i+8}, v_{i+9} \in \mathcal{S}$ in order for all pairs of vertices to be differentiated. 

Suppose $v_{i+2}, v_{i+3} \in \mathcal{S}$. If $v_{i-2} \not\in \mathcal{S}$, then $v_i$ and $v_{i+4}$ remain undifferentiated. If $v_{i-1} \not\in \mathcal{S}$, then $v_{i+1}$ and $v_{i+5}$ remain undifferentiated. If either of $v_{i+8}$ or $v_{i+9}$ is not in $\mathcal{S}$, then there is an unseen vertex. Now consider when $v_{i+3}, v_{i+4} \in \mathcal{S}$. Then $v_{i+2}$ and $v_{i+5}$ remain undifferentiated even if $v_{i-2}, v_{i-1}, v_{i+8}, v_{i+9} \in \mathcal{S}$. The other possibilities for $\mathcal{B} \cap \mathcal{S}$ can be handled similarly.     
\end{proof}

\begin{defn}
Within a grid graph $G_{2, m}$ with locating-dominating set $\mathcal{S}$, if we have an $(2\times 4)$-block $\mathcal{B}$ such that $|\mathcal{B} \cap \mathcal{S}| \geq 3$, we call such a block a $(2\times 4)$-\textit{small block}. Similarly, if we have a $(2\times 5)$-block $\mathcal{B}$ such that $|\mathcal{B} \cap \mathcal{S}| = 4$, we call such a block a $(2\times 5)$-\textit{large block}.
\end{defn}

\begin{lemma}\label{lower}
For $n \geq 2$, we have that $LD(P_2 \square P_n) \geq \lceil (3n-1)/4 \rceil.$
\end{lemma}

\begin{proof}
Given a locating-dominating set $\mathcal{S}$ for $G_{2, m}$, we describe an algorithm for dividing $G_{2, m}$ into a sequence of $(2\times 4)$-small blocks and $(2\times 5)$-large blocks, plus possibly a small residual block of size at most $(2 \times 3)$ at the end. We will work left to right on $G_{2, m}$, first considering the first 8 vertices of $G_{2, m}$. By Lemma \ref{twovertices}, at least 3 of the first 8 vertices must be in $\mathcal{S}$, so group the first 8 vertices into a $(2\times 4)$-small block. Now as long as there are at least 8 vertices left, iterate: 

\begin{enumerate}[1.]
    \item If at least 3 of the next 8 vertices are in $\mathcal{S}$, then form the 8 vertices into a $(2\times 4)$-small block.
    \item If not, then exactly two of the next 8 vertices are in $\mathcal{S}$. Therefore, we can apply Lemma \ref{twovertices} and form the next 10 vertices into a $(2\times 5)$-large block. 
\end{enumerate}
Let $\delta$ be the width of the residual block. Note that $0 \leq \delta \leq 3$. Additionally, let $r$ be the number of vertices in the residual block that are also in $\mathcal{S}$. We claim that $r \geq (3\delta-1)/4$. To confirm this, we perform casework on the possible values of $\delta$:
\begin{enumerate}[1.]
    \item If $\delta = 0$, then $r = 0$ as well. However, $(3\delta-1)/4 = -1/4$, so we are done.
    \item Now we consider when $\delta=1$. Assume, for the sake of contradiction, that $r=0$. Consider the following cases:
        \begin{enumerate}
            \item First consider when the last block before the residual block is a $(2\times 4)$-small block. Then note that choosing any configuration of 3 vertices in the $(2\times 4)$-small block to be in the locating-dominating set either leaves two vertices undifferentiated or results in an unseen vertex. Now consider when there are $3+x$, for $1\leq x \leq 5$, vertices that are both in the locating-dominating set and also in the $(2\times 4)$-small block. Then clearly $r+x \geq (3\delta-1)/4$, so we are done. 
            \item Next consider when the last block before the residual block is a $(2\times 5)$-large block. The key observation here is that Lemma \ref{twovertices} tells us that the two vertices immediately preceding the $(2\times 5)$-large block and the last two vertices of the $(2\times 5)$-block must all be in the locating-dominating set (see Figure \ref{fig:tricksitch}). Therefore, using casework to analyze the positions of the remaining two vertices in the $(2\times 5)$-large block, we get that we need at least one vertex in the residual block to be in the locating-dominating set.   
        \end{enumerate}
    This is a contradiction, so $r\geq 1$. Since $1 \geq (3\cdot 1-1)/4$, our desired bound holds. 
    \item When $\delta=2$, we once again have cases. Note that $r\geq 1$ since otherwise we have unseen vertices. Assume, for the sake of contradiction, that $r=1$. Performing similar casework as in the $\delta=1, r=0$ case, we see that this case is not possible either. Therefore, $r\geq 2$. Then we have that $2 \geq (3\cdot 2-1)/4$, so our bound holds.
    \item When $\delta=3$, we must have that $r\geq 2$ by Lemma \ref{twovertices}. Then $2 \geq (3\cdot 3-1)/4$, as desired.
\end{enumerate}

Now let $a$ be the number of $(2\times 4)$-small blocks and $b$ be the number of $(2\times 5)$-large blocks produced by this process. We have the following:
\begin{align*}
    |\mathcal{S}| &\geq 3a+4b+r \\
    n &= 4a+5b+\delta.
\end{align*}
Manipulating the latter equation, we get that $$\frac{3n-1}{4} = 3a+\frac{15}{4}b+\frac{3\delta-1}{4}.$$ Since we know that $r \geq (3\delta-1)/4$, this gives us the desired bound. 
\end{proof}

\begin{figure}[h]
    \centering
    \includegraphics{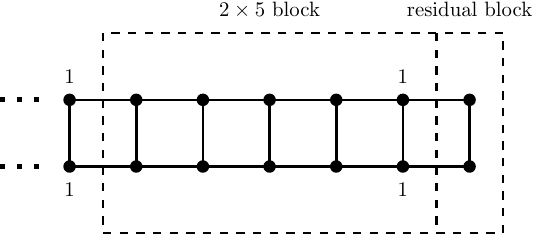}
    \caption{An illustration of the vertices required to be in the locating-dominating set}
    \label{fig:tricksitch}
\end{figure}

The next step is establishing a matching upper bound. In order to do so, we first give names to commonly used $(2\times 4)$-blocks in \Cref{fig:labellings}. Using these $(2\times 4)$-blocks, which are not necessarily adjacency resolving broadcasts, we can construct larger adjacency resolving broadcasts. This is especially useful when the adjacency resolving broadcast has many repeating elements. 

\begin{defn}
Given a block $\mathcal{B}_1$ and another block $\mathcal{B}_2$, we let $\mathcal{B}_1 \mathcal{B}_2$ denote the block which is $\mathcal{B}_1$ with $\mathcal{B}_2$ appended to it. In other words, we put the two blocks next to each other and then add a sufficient number of edges between to connect the two blocks. We call this action $\textit{concatenation}$. Using the common blocks defined in \Cref{fig:labellings}, we can create the concatenation $AB$, as shown in \Cref{fig:combine}.
\end{defn}

\begin{figure}[h]
    \centering
    \includegraphics[width=0.5\textwidth]{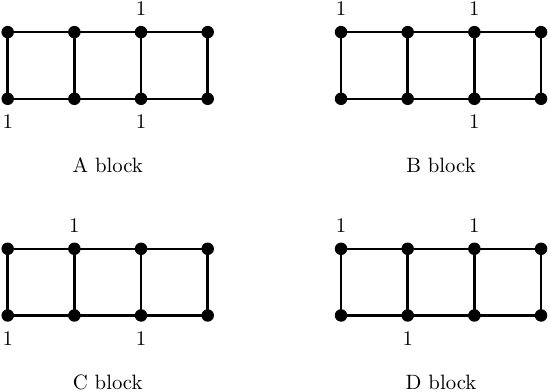}
    \caption{Common $(2\times4)$-blocks}
    \label{fig:labellings}
\end{figure}

\begin{figure}[h]
    \centering
    \includegraphics[width=0.5\textwidth]{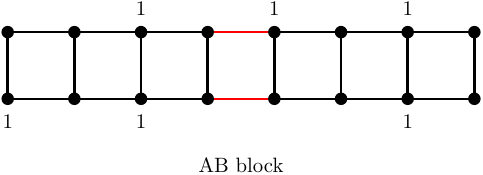}
    \caption{Appending blocks}
    \label{fig:combine}
\end{figure} 

\begin{lemma}\label{upper}
For $n \geq 2$, we have that $\adim(P_2 \square P_n) \leq \lceil (3n-1)/4 \rceil.$
\end{lemma}

\begin{proof}
We provide a construction to show an upper bound. In particular, we break up the construction into 8 cases, depending on the residue of $n \bmod 8$. Each case is constructed by repeating a block of length 8 and then adding on a ``tail." The repeated block of 8 is constructed from the blocks given in \Cref{fig:labellings}, so now we just describe the tails. There are only five possible tails and we have listed them in \Cref{fig:tails}. 

Now that we have defined our blocks and tails, we can give the construction for each of the eight residue classes, as in \Cref{constable}. Note that each of these constructions is just made up of some concatenation combination of the set of blocks $\{AB, CD, T_{1}, T_2, T_{3}, C, A\widetilde{T_1}, CT_2, A\widetilde{T_3}\}$, all of which are resolved. Only the blocks in $\{C, T_3, CT_2, A\widetilde{T_3}\}$ have an one unseen vertex. However, this is not an issue since none of these blocks are ever in the same concatenation sequence for a construction, so we are done.    
\end{proof}

\begin{figure}[h]
    \centering
    \includegraphics[width=\textwidth]{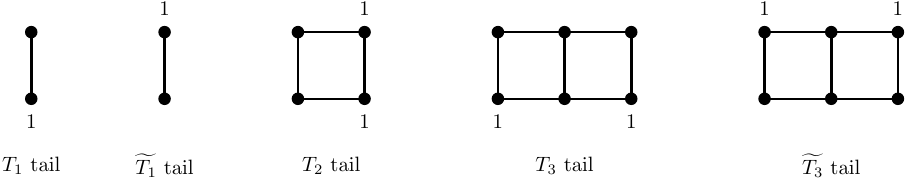}
    \caption{Set of all possible tails}
    \label{fig:tails}
\end{figure}

\begin{table}[h]
    \centering
    \begin{tabular}{ |c|c| }
    \hline
    $n \bmod 8$ & Construction of $G_{2,n}$  \\ 
    \hline
    0 & $(CD)^{n/8}$  \\ 
    1 & $(AB)^{(n-1)/8}T_1$  \\ 
    2 & $(CD)^{(n-2)/8}T_2$  \\
    3 & $(AB)^{(n-3)/8}T_3$  \\
    4 & $(CD)^{(n-4)/8}C$  \\
    5 & $(AB)^{(n-5)/8}A\widetilde{T_1}$  \\
    6 & $(CD)^{(n-6)/8}CT_2$  \\
    7 & $(AB)^{(n-7)/8}A\widetilde{T_3}$  \\
    \hline
    \end{tabular}
    \caption{All constructions}
    \label{constable}
\end{table}

\begin{ex}
See \Cref{fig:allexamples} for an example of a construction. Unlabeled vertices are not in the adjacency resolving broadcast. Notice the common repeated block of length 8 on the left hand side of each construction. It is easy to check that the repeating block, each tail, and the block-block and block-tail overlaps are all differentiated.  
\end{ex}

\begin{figure}[h]
    \centering
    \includegraphics[width=\textwidth]{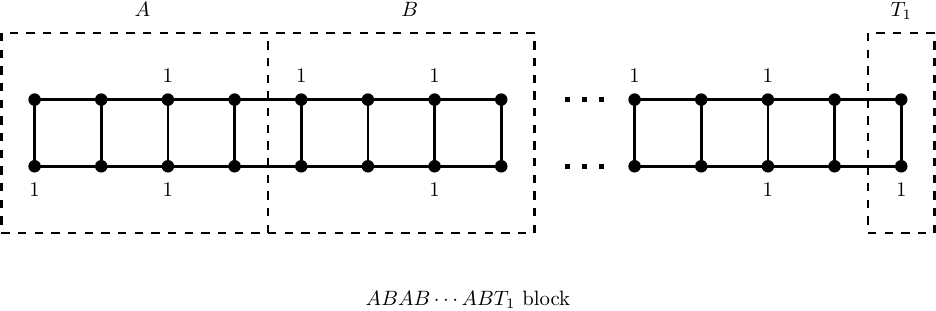}
    \caption{A sample construction}
    \label{fig:allexamples}
\end{figure}

Now we are able to prove Theorem \ref{2block}, on the adjacency dimension of $P_2\square P_n$.

\begin{proof}[of Theorem \ref{2block}]
The proof of this theorem follows from Lemmas \ref{lower} and \ref{upper}. 
\end{proof}

\subsection{Adjacency Dimension of $G_{3, m}$}

\begin{lemma}\label{beforeafter3}
Let $\mathcal{S}$ be a locating-dominating set for $G_{3, m}$ and let $\mathcal{B}$ be a $(3 \times 3)$-block. Then $|\mathcal{B} \cap \mathcal{S}| \geq 2$.
\end{lemma}

\begin{proof}
Let $\mathcal{B} = \{v_i, \dots v_{i+8}\}$. If instead $|\mathcal{B} \cap \mathcal{S}| = 1$, then any positioning of the vertex $v \in \B\cap\s$ either leads to $v_3$ and $v_5$ being undifferentiated vertices, or one of them being unseen. 
\end{proof}

\begin{figure}[h]
    \centering
    \includegraphics{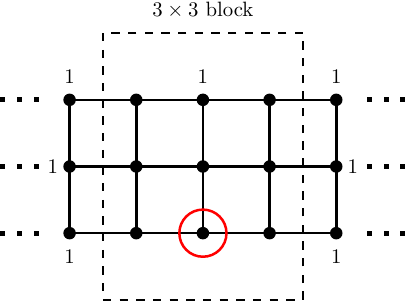}
    \caption{A $3\times3$ block with only 1 vertex in the resolving set has an unseen vertex (circled in red)}
    \label{fig:example3}
\end{figure}

\begin{lemma}\label{onevertex}
Let $\mathcal{S}$ be a locating-dominating set for $G_{3, m}$. If there exists an $(3 \times 3)$-block $\mathcal{B}$ within $G_{3, m}$ consisting of the vertices $v_i, \dots, v_{i+8}$ such that $|\mathcal{B} \cap \mathcal{S}| = 2$, then $i \neq 1, i+8 \neq 3m$, and $v_{i-3}, v_{i-2}, v_{i-1}, v_{i+9}, v_{i+10}$ and $v_{i+11}$ must all be in the locating-dominating set of $G_{3, m}$. 
\end{lemma}

\begin{proof}
This is finite case check. We consider each placement of the vertices in $\B\cap\s$ and see that all the vertices surrounding the block must be in the adjacency resolving set for the graph to be resolved. 
\end{proof}

\begin{defn}
Within a grid graph $G_{3, m}$ with locating-dominating set $\mathcal{S}$, if we have an $(3\times 3)$-block $\mathcal{B}$ such that $|\mathcal{B} \cap \mathcal{S}| \geq 3$, we call such a block a $(3\times 3)$-\textit{small block}. Similarly, if we have an $(3\times 4)$-block $\mathcal{B}$ such that $|\mathcal{B} \cap \mathcal{S}| = 5$, we call such a block a $(3\times 4)$-\textit{large block}. 
\end{defn}

\begin{lemma}\label{lowerthree}
We have that $$LD(P_3 \square P_n) \geq 
\begin{cases} 
    n+1 & \text{if } n \equiv 1 \bmod 3 \\
    n & \text{otherwise.}
\end{cases}$$
\end{lemma}

\begin{proof}
We proceed similarly to the proof of Lemma \ref{lower} by presenting an algorithm. However, in this case, we instead divide our given locating-dominating set $\mathcal{S}$ into a sequence of $(3\times 3)$-small blocks and $(3\times 4)$-large blocks along with some residual block. Working left to right, we see that at least 3 of the first 9 vertices must be in $\mathcal{S}$. Therefore, we group the first 9 vertices into a $(3\times 3)$-small block. Now we iterate on the following process as long as there are at least 9 vertices left:

\begin{enumerate}[1.]
    \item If at least 3 of the next 9 vertices are in $\mathcal{S}$, then form the 9 vertices into a $(3\times 3)$-small block.
    \item If exactly 2 of the next 9 vertices are in $\mathcal{S}$, then we can apply Lemma \ref{onevertex} to see that 5 of the next 12 vertices are in $\mathcal{S}$. Therefore, we can group the next 12 vertices into a $(3\times 4)$-large block.
\end{enumerate}
Then when we have fewer than 9 vertices left, we finish using similar casework regarding the residual block as we did in Lemma \ref{lower}. Note that the lemma roughly tells us that we need at least 1 weight per column of $G_{3, m}$. However, the key observation is that using a $(3\times 4)$-large block forces that count to increase by 1, as we need 5 vertices in the resolving set to completely resolve a $(3\times 4)$-block which has 4 columns.    
\end{proof}

\begin{lemma}\label{upperthree}
We have that $$\adim(P_3 \square P_m) \leq 
\begin{cases} 
    m+1 & \text{if } m \equiv 1 \bmod 3 \\
    m & \text{otherwise.}
\end{cases}$$
\end{lemma}

\begin{proof}
As before, we provide a construction to show an upper bound. Here we break up the construction into 3 cases, based on whether $m$ is $0, 1$ or $2 \bmod 3$. The constructions are done by concatenating blocks and tails, but in this case we have only one block and two tails (see Figure \ref{fig:GBlock}). Each residue class corresponds to a different tail. 

Now we give the explicit construction for each of the three cases in Table \ref{constable3}. Note that while $G$ and $G_2$ are resolved, each has one unseen vertex. However, note that concatenating $G$ to itself resolves the unseen vertex. Similarly, concatenating $G_2$ to $G$ also resolves the unseen vertex. Therefore, we always have at most one unseen vertex. Thus we are done.
\end{proof}

\begin{figure}[h]
    \centering
    \includegraphics{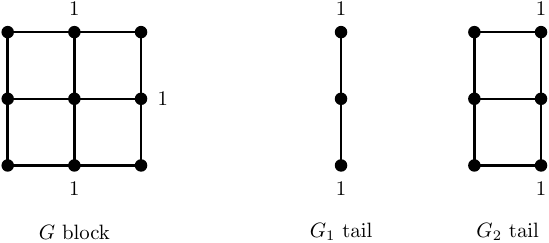}
    \caption{Blocks and tails used for upper bound construction}
    \label{fig:GBlock}
\end{figure}

\begin{table}[h]
    \centering
    \begin{tabular}{ |c|c| }
    \hline
    $m \bmod 3$ & Construction of $G_{3,m}$  \\ 
    \hline
    0 & $G^{m/3}$  \\ 
    1 & $G^{(m-1)/3}G_1$  \\ 
    2 & $G^{(m-2)/3}G_2$  \\
    \hline
    \end{tabular}
    \caption{All constructions}
    \label{constable3}
\end{table}

Now we are able to show Theorem \ref{3block}, on the adjacency dimension of $P_3\square P_n$.

\begin{proof}[of Theorem \ref{3block}]
The proof of this theorem follows from Lemmas \ref{lowerthree} and \ref{upperthree}.
\end{proof}

\section{Directed graphs}\label{dirgraph}

Recall from previous sections that adjacency and broadcast dimension are defined for finite, simple, and undirected graphs. Here we extend these notions to directed graphs as well. In the robot analogy from \cite{khuller1996landmarks}, we had a robot navigating along the edges of some graph $G$ and there existed landmarks that the robot could sense its distance from. However, in this case, the signal from the landmark can travel only in the direction of the directed edges. 

Let $\vec{G} = (V(\vec{G}), E(\vec{G}))$ be a finite, simple, and directed graph. Then let $d(u, v)$ for $u, v \in V(\vec{G})$ be the length of the shortest directed path from $u$ to $v$ if one exists or be infinity if one does not exist. As before, we let $d_k(u, v) \coloneqq \min (d(u, v), k+1),$ where $u, v \in V(\vec{G})$. This then allows us to write down the following definitions:

\begin{defn}
Let $\vec{G}$ be a directed graph. Then we call a subset $A \subseteq V(\vec{G})$ an \textit{adjacency resolving set} of $G$ if for any distinct $x, y \in V(\vec{G})$ there is a vertex $z \in A$ such that $d_1(z, x) \neq d_1(z, y)$. The \textit{adjacency dimension} $\adim(\vec{G})$ of $\vec{G}$ is the minimum cardinality of an adjacency resolving set of $\vec{G}$.     
\end{defn}

\begin{defn}
A function $f : V(\vec{G}) \rightarrow \Z^+ \cup \{0\}$ is a \textit{resolving broadcast} of $\vec{G}$ if, for any distinct $x, y \in V(\vec{G})$, there is a vertex $z \in \supp(f) \coloneqq \{ v\in V(\vec{G}) : f(v) > 0\}$ such that $d_{f(z)}(z, x) \neq d_{f(z)}(z, y)$. Additionally, the \textit{broadcast dimension} $\bdim(\vec{G})$ of $\vec{G}$ is the minimum of $\sum_{v \in V(\vec{G})}f(v)$ over all resolving broadcasts $f$ of $\vec{G}$.  
\end{defn}

We define the adjacency resolving broadcast of a directed graph analogously to the non-directed graph case.  

Note that orienting an undirected graph $G$ to obtain a directed graph $\vec{G}$ can result in $\bdim(\vec{G})$ being smaller or larger than $\bdim(G),$ as we can see in Figure \ref{fig:directed_comparison}. However, the question is, how much can these two broadcast dimensions differ?

\begin{figure}[h]
    \centering
    \includegraphics{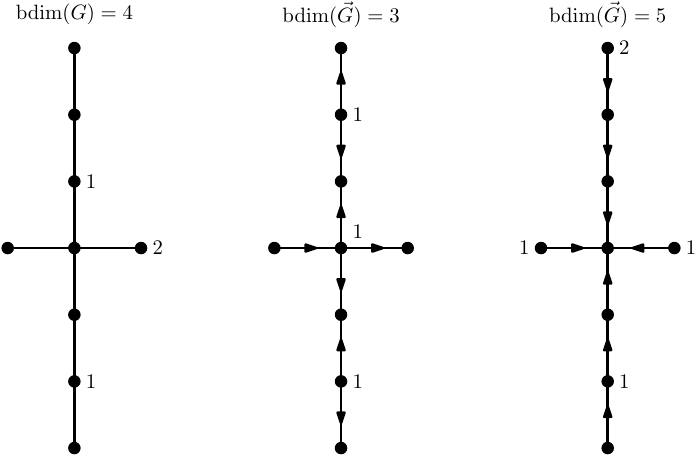}
    \caption{Result of adding direction on broadcast dimension}
    \label{fig:directed_comparison}
\end{figure}

In order to help with the following two theorems, we describe a family of graphs $F_k$ from \cite{zubrilina2018edge}, which is also described in \cite{geneson2020broadcast}. Here we present the latter formulation. We construct a graph $F_k = (V(F_k), E(F_k))$ such that $|V(G)| = k+2^k$. We start with $k$ vertices $v_1,\dots, v_k$ and form a complete graph. Then we take $2^k$ new vertices $\{u_b\}_{b\in\{0, 1\}^k}$ and also form a complete graph on these vertices. Note that the indices of $u_b$ are binary strings of length $k$. Now we add an edge in $G$ between $u_b$ and $v_j$ if and only if the $j$th digit of $b$ is 1. 

\begin{thm}\label{verybig}
For all $k$, there exist graphs $G$ and orientations $\vec{G}$ of $G$ such that $k \geq \adim(G) \geq \bdim(G)$ and $\adim(\vec{G}) \geq \bdim(\vec{G}) \geq 2^k.$ 
\end{thm}

\begin{proof}
We utilize the graph family $F_k$. Note that $\adim(F_k) \leq k$, since we can place weights of 1 on $v_1,\dots, v_k$. Now consider the directed version of this family of graphs where edges connected to any of the $k$ vertices are directed towards the $k$ vertices and all other vertices are directed arbitrarily. Then clearly $\bdim(\vec{F_k}) \geq 2^k$ since every $u_b$ must have a weight which is greater than or equal to 1. Therefore we are done.  
\end{proof}

\begin{figure}[h]
    \centering
    \includegraphics{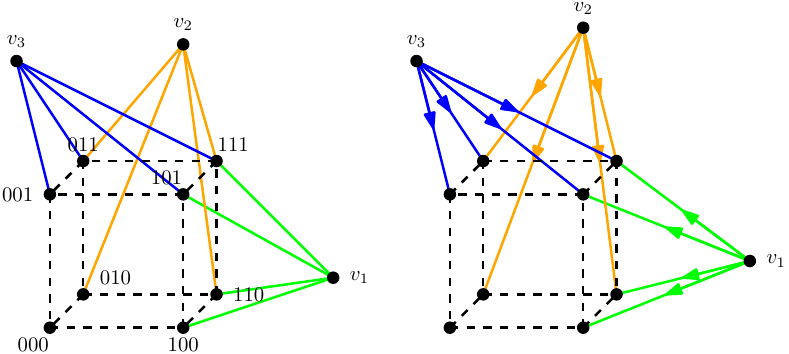}
    \caption{On the left we have an illustration of the relevant edges of $F_3$. The image on the right shows the edges we direct away from $v_i$ in $\vec{R_3}$. }
    \label{fig:my_label}
\end{figure}

\begin{thm}\label{verysmall}
For all $k$, there exist graphs $G$ and orientations $\vec{G}$ of $G$ such that $\adim(G) \geq \bdim(G) = 2^k+k-1$ and $k\geq \adim(\vec{G}) \geq \bdim(\vec{G}).$ 
\end{thm}

\begin{proof}
Inspired by a family of graphs $F_k$, we introduce a new family of graphs $R_k$ here. First, we start with the vertex set of $F_k$ and then create the complete graph on those $2^k+k$ vertices. Recalling that the broadcast dimension of the complete graph $K_n$ is $n-1$, we have that $\bdim(R_k) = 2^k+k-1$. 

Now consider a directed version of this family of graphs, $\vec{R_k}$. The edge between $v_i$ and $u_b$ is oriented towards $v_i$ if the $i$th digit of $b$ is 1, and oriented away from $v_i$ is the $i$th digit of $b$ is 0. Note that in this case $\adim(\vec{R_k}) \leq k$, since we just place weights of 1 on all of $v_1, \dots, v_k$.
\end{proof}

\begin{rem}
Note that if $|V(G)| = n$, then $\log n \ll \adim(G), \adim(\vec{G}) \leq n-1$, so the ``gap" between $\adim(G)$ and $\adim(\vec{G})$ is at most logarithmic.
\end{rem}

\begin{thm}\label{layers}
Let $\vec{T_k}(n)$ be a complete $k$-ary tree with $n$ layers such that every edge is directed away from the root. Then $$\adim(\vec{T_k}(n)) = \begin{cases}
    1+k^2+\dots+k^{n-3}+k^{n-1} & \text{for odd } n\\
    k+k^3+\dots+k^{n-3}+k^{n-1} & \text{for even } n.
\end{cases}$$  
\end{thm}

\begin{proof}
We begin by considering where we need to place weights on the graph to resolve all vertices. We have two cases for placing weights on the children of the current vertex we are on:

\begin{enumerate}
    \item If the current vertex we are on has a weight of 0, then all of its children have to have weights of 1 since otherwise they would not be differentiated from each other or would become unseen. 
    \item If the current vertex we are on has a weight of 1, then all but one of its children must have a weight of 1. In this case we are allowed to have one child with no weight since it becomes uniquely differentiated by its parent when all of its sibling have their own weight.  
\end{enumerate}

Now let $\mu_i$ denote the number of $1$'s in layer $i$ and let $\nu_i$ denote the number of 0's in layer $i$. Clearly $\nu_i = k^{i-1}-\mu_i$. Now applying the two rules from above, we can write that
\begin{align*}
    \mu_i &\geq (k-1)\mu_{i-1}+k\nu_{i-1} \\
    &\geq (k-1)\mu_{i-1}+k(k^{i-1}-\mu_i) \\
    &\geq k^i-\mu_{i-1}.
\end{align*}

Therefore, we have that $\mu_i+\mu_{i-1}\geq k^i$. Adding this inequality for $i = n-1, n-3, n-5, \dots$ gives us the desired result.

The construction which exactly acquires this adjacency dimension value is as follows. When $n$ is odd, we place weights on all vertices at depths $0, 2, 4, \dots$. When $n$ is even, we place weights on all vertices at depths $1, 3, 5, \dots$.
\end{proof}

Now we prove Theorem \ref{allthesame}. We first introduce a definition in order to help with the proof. 

\begin{defn}
A vertex $v$ is \textit{crucial} for a vertex $u$ if there exists a vertex $w$ such that $v$ resolves the pair $u$ and $w$ but no other vertex does.
\end{defn}

\begin{proof}[of Theorem \ref{allthesame}]
Starting with a resolving broadcast $f$ for $\vec{G}$, we show how to transform it into an adjacency resolving broadcast $g$. We traverse our tree starting from the root and consider the vertices which are a distance $1, 2, 3, \dots$ away. As we traverse $\vec{G}$, let us consider the first vertex with weight greater than 1 that we pass. Call this vertex $v$ and let it have weight $w$. Let $C$ be the set of vertices for which $v$ is crucial. (Note that $v$ may or may not be in $C$.) Note that $|C|\leq w+1$ since there can only be one vertex at each layer of the graph for which $v$ is crucial. 

When $|C| = w+1$, we create a modified resolving broadcast $f_1$ for $\vec{G}$ from $f$. We remove the weight of $w$ from $v$ and then place weights of 1 on all vertices in $C$, except for the vertex that is a child of $v$. Note that the vertices in $C$ previously must have not had weights on them since otherwise vertex $v$ would not be crucial for them. Through this process, all vertices remain differentiated. 

When $|C| \leq w$, we once again create a modified resolving broadcast $f_1$ for $\vec{G}$ from $f$, but the process of doing so is slightly different. As before, we again remove the weight of $w$ from $v$ then we place weights of 1 of all vertices in $C$ in this case. 

As we traverse the vertices of $\vec{G}$, we stop at each vertex with weight greater than 1 and create a modified resolving broadcast. Note that the each modified broadcast eliminates a weight greater than 1 and replaces it with weights of 1. Therefore, through traversing the vertices of $\vec{G}$, we get a sequence of modified resolving broadcasts and the total weight never increases. Therefore, after traversing all vertices, we are left with a resolving broadcast that just consists of 1's. Thus this is an adjacency resolving broadcast, so we are done.  

Note that this process terminates since we only have a finite number of vertices.  
\end{proof}

\section{Acknowledgements}
This research was conducted at the 2021 Duluth Combinatorics REU. The author thanks Joe Gallian for organizing the REU and her advisors Amanda Burcroff, Noah Kravitz, Colin Defant and Yelena Mandelshtam. The author also thanks Emily Zhang, Evan Chen, and Alan Peng for helpful discussions, Mitchell Lee and Noah Kravitz for guiding her through paper revisions, as well as her peers at Duluth for creating a supportive environment. Lastly, the author thanks her grandmother Geetha Krishnamurthy for her continued support of the author's mathematical endeavors. This research was funded by NSF-DMS Grant 1949884 and NSA Grant H98230-20-1-0009, with additional support from the MIT CYAN Mathematics Undergraduate Activities Fund.

\bibliographystyle{plainnat}
\bibliography{sample-dmtcs}

\end{document}